\documentclass[a4paper]{amsart}
\usepackage{amssymb}
\usepackage[matrix,arrow,curve,tips]{xy}\SelectTips{cm}{}
\usepackage{mathrsfs}

\newtheorem{dfn}{Definition}
\newtheorem{prop}[dfn]{Proposition}
\newtheorem{theo}[dfn]{Theorem}
\newtheorem{cor}[dfn]{Corollary}
\newtheorem{ex}[dfn]{Example}
\newtheorem{lem}[dfn]{Lemma}

\newcommand{\Rem}{\noindent {\it Remark. }}
\newcommand{\oo}{\,\mbox{-}\,}
\newcommand{\com}{\circ}

\newcommand{\cF}{\mathord{\mathcal{F}}}
\newcommand{\cO}{\mathord{\mathcal{O}}}
\newcommand{\cN}{\mathord{\mathcal{N}}}

\newcommand{\src}{\mathord{\mathrm{s}}}
\newcommand{\trg}{\mathord{\mathrm{t}}}

\newcommand{\id}{\mathord{\mathrm{id}}}

\newcommand{\Ker}{\mathord{\mathrm{Ker}}}

\newcommand{\I}{\mathord{\mathrm{I}}}
\newcommand{\K}{\mathord{\mathrm{K}}}
\newcommand{\Nor}{\mathord{\mathrm{N}}}
\newcommand{\Tan}{\mathord{\mathrm{T}}}
\newcommand{\Ad}{\mathord{\mathrm{Ad}}}
\newcommand{\GL}{\mathord{\mathrm{GL}}}
\newcommand{\Aut}{\mathord{\mathrm{Aut}}}
\newcommand{\Inn}{\mathord{\mathrm{Inn}}}
\newcommand{\Eff}{\mathord{\mathrm{Eff}}}

\newcommand{\W}{\mathord{\mathrm{W}}}
\newcommand{\U}{\mathord{\mathrm{U}}}
\newcommand{\Normalizer}{\mathord{\mathrm{N}}}
\newcommand{\Centralizer}{\mathord{\mathrm{C}}}

\newcommand{\GG}{\mathscr{G}}
\newcommand{\HH}{\mathscr{H}}
\newcommand{\KK}{\mathscr{K}}

\newcommand{\BB}{\mathscr{B}}

\newcommand{\Gpd}{\mathord{\mathsf{Gpd}}}
\newcommand{\GPD}{\mathord{\mathsf{GPD}}}
\newcommand{\hGPD}{\mathord{\mathsf{hGPD}}}

\begin{document}

\title[]%
      {Homotopy sequence of a topological groupoid with a basegroup
       and an obstruction to presentability of proper regular Lie groupoids}

\author{B. Jelenc}
\address{Institute of Mathematics, Physics and Mechanics,
         University of Ljubljana, Jadranska 19,
         1000 Ljubljana, Slovenia}
\email{blaz.jelenc@imfm.si}

\author{J. Mr\v{c}un}
\address{Department of Mathematics, University of Ljubljana,
         Jadranska 19, 1000 Ljubljana, Slovenia}
\email{janez.mrcun@fmf.uni-lj.si}

\thanks{This work was supported in part by
        the Slovenian Research Agency (ARRS) project J1-2247.}

\begin{abstract}
A topological groupoid $\GG$ is $K$-pointed,
if it is equipped with a homomorphism from a topological group
$K$ to $\GG$. We describe the homotopy groups of such
$K$-pointed topological groupoids and relate these groups to
the ordinary homotopy groups in terms of a long exact sequence.
As an application, we give an obstruction to presentability
of proper regular Lie groupoids.
\end{abstract}

\maketitle

Topological and Lie groupoids are geometric objects that can be used
to represent various singular geometric structures. Examples of such
structures include foliations, equivalence relations, group actions, 
orbifolds, etc. The topological groupoids associated to these geometric structures
are usually determined up to Morita equivalence.
Homotopy groups are examples of Morita invariants of topological groupoids.
They can be defined as the homotopy groups of the associated classifying space
(see \textit{e.g.}\ \cite{Haefliger,Moerdijk1995}).
An equivalent description of these groups
is given by the group of homotopy classes of principal $(\GG,a)$-bundles over
$(I^n,\partial I^n)$ \cite{Haefliger,Jelenc,Mrcun1996}. Here $(\GG,a)$ is a
pointed topological groupoid with basepoint $a$ in the space of objects
$\GG_0$ of $\GG$.
This approach has several advantages, for
instance, it enables us to describe the notion of a Serre fibration in the Morita
category of topological groupoids and derive the associated long exact
sequence of homotopy groups (\cite{Jelenc}, see also 
\cite{Noohi} for a study of fibrations in the context of stacks).

The description of homotopy groups in terms of principal bundles over
$(I^n,\partial I^n)$ also allows us to consider the homotopy groups
of a topological groupoid with a more general base
instead of a simple basepoint in its space of objects.
A homomorphism $\rho$ from a locally compact Hausdorff topological group
$K$ to a topological groupoid $\GG$ defines a {\em $K$-basegroup}
in $\GG$ at a point $a$ in $\GG_0$.
We say that such a groupoid with a $K$-basegroup is {\em $K$-pointed}
and denote it by the triple $(\GG,a,\rho)$.
The $K$-pointed topological groupoids form
a Morita category, where morphisms are given by the (isomorphism classes of)
principal bundles. Such a principal $(\GG,a,\rho)$-bundle over $(\HH,b,v)$
is a pointed space $(P,p)$ which is a principal $\GG$-bundle over $\HH$
(\cite{Haefliger1984,Mrcun1996,MoerdijkMrcun2005})
compatible with the action of the basegroups.
We can extend this Morita category further to include
{\em $K$-marked} topological groupoids, which are triples
$(\GG,A,\rho)$, where $A$ is a subset of $\GG_0$ and $\rho$ is a continuous choice
of a $K$-basegroup at each $a\in A$.
For instance, this Morita category of $K$-marked topological groupoids
includes $(K\times I^n,\partial I^n,\tau)$, where $\tau$ is the  inclusion.
We define the $n$-th homotopy group
of a $K$-pointed topological groupoid $(\GG,a,\rho)$ as the group of homotopy
classes of principal $(\GG,a,\rho)$-bundles over $(K\times I^n,\partial I^n,\tau)$,
and denote it by $\pi_n^K(\GG)=\pi_n^K(\GG,a,\rho)$. It is clear that this gives us
a more refined
version of homotopy groups, as the elements of $\pi_n^K(\GG)$ can be
seen as elements of the ordinary homotopy group
$\pi_n(\GG)=\pi_n(\GG,a)$ with an additional
left $K$-action.
Similar constructions are known in the context of orbifolds
\cite{Chen} and orbispaces \cite{HenriquesGepner}.

Next, we associate to every topological groupoid $\GG$
the groupoid $\GG^K$, which is the action groupoid associated
to the adjoint $\GG$-action 
on the space of all morphisms from $K$ to $\GG$. This groupoid
plays the role of the exponential in the Morita category of topological groupoids.
This exponential provides an alternative description of
the homotopy groups of a $K$-pointed topological groupoid $(\GG,a,\rho)$
as ordinary homotopy groups of the pointed groupoid $(\GG^K,\rho)$.
We also extend the notion of a Serre fibration to the notion of a {\em Serre
$K$-fibration}. We are especially interested in the canonical functor
$\omega:\GG^K\to\GG$ that gives the relation between homotopy groups of a
$K$-pointed groupoid $(\GG,a,\rho)$ and ordinary homotopy groups of $(\GG,a)$.
When $\omega$ is a Serre fibration, this relation is expressed in terms of a long
exact sequence
$$
\ldots\to\pi_n(\hom(K,\GG_a))\to\pi_n^K(\GG)\to\pi_n(\GG)
      \to\pi_{n-1}(\hom(K,\GG_a))\to\ldots\;,
$$
where $\hom(K,\GG_a)$ is the space of all homomorphisms from $K$ to the isotropy
group $\GG_a$ of $\GG$ at $a$.

In the last part of the paper,
we study the homotopy groups of $K$-pointed connected proper
regular Lie groupoids. Choosing the $K$-basegroup of such a groupoid to be given
by the Lie group $\K_a(\GG)$ of its ineffective arrows at $a\in\GG_0$,
the long exact sequence associated to the functor $\omega$
gives us the {\em monodromy map}
$$
\partial:\pi_1(\GG)\to\pi_0(\Aut(\K_a(\GG)))\;,
$$
which is a Morita invariant of $\GG$.
This map provides us with an obstruction
to presentability of connected proper regular Lie groupoids as global quotients
(see Theorem \ref{obstruction}).
We show that if such a groupoid $\GG$ is Morita equivalent
to the translation groupoid associated to an action of a connected
compact Lie group on a smooth manifold, then there exists
a faithful unitary representation of $\K_a(\GG)$ such that
the corresponding character is fixed under precomposition
with any automorphism representing a class in the image
of the monodromy map $\partial:\pi_1(\GG)\to\pi_0(\Aut(\K_a(\GG)))$.

\section{Morita category of group-marked topological groupoids}

\subsection{Topological groupoids}
To fix the notations, let us first recall some
basic definitions and facts (for details and examples, see \textit{e.g.}\
\cite{Mackenzie,MoerdijkMrcun2003,MoerdijkMrcun2005}). 
A {\em groupoid} is a small category in which
all arrows are invertible. In particular, a groupoid $\GG$
has its set of {\em objects} $\GG_0$ and its set of {\em arrows}
$\GG_1$. Each arrow $g$ of $\GG$ has its {\em source} object $\src(g)$ and 
its {\em target} object $\trg(g)$.
For two objects $x,x'\in\GG_0$ we denote by
$\GG(x,x')=\src^{-1}(x)\cap\trg^{-1}(x')$ the set of arrows
{\em from} $x$ {\em to} $x'$ in $\GG$.
The {\em product} arrow $gg'\in\GG_1$ of two arrows
$g,g'\in\GG_1$ is defined if $\src(g)=\trg(g')$, and in this case
$\src(gg')=\src(g')$ and  $\trg(gg')=\trg(g)$.
Each object $x$ has its {\em unit} arrow $1_x\in\GG(x,x)$,
and each arrow $g$ has its {\em inverse} arrow $g^{-1}$.
A groupoid $\GG$ is sometimes denoted by $(\GG_1\rightrightarrows\GG_0)$.
The set $\GG(x,x)$ is a group, called the {\em isotropy group}
of $\GG$ at the object $x$, and denoted also by $\GG_x$ or $\I_x(\GG)$.

A {\em topological groupoid} is a groupoid $\GG$ with topology
on $\GG_0$ and on $\GG_1$ such that all the structure maps
mentioned above are continuous. In particular, the isotropy
groups of a topological groupoid are topological groups.
A {\em homomorphism} of topological groupoids is a continuous functor
between them. In this way we get the category of topological
groupoids, which we denote by $\Gpd$.

First examples of topological groupoids are topological spaces
and topological groups:
any topological space $X$ can be viewed as the (unit) topological groupoid
$(X\rightrightarrows X)$, and any topological group $G$ is a topological groupoid
with only one object $(G\rightrightarrows \ast)$. The category
of topological spaces and the category of topological groups are both
full subcategories of $\Gpd$.

\subsection{Actions of topological groupoids}
Recall that a right action of a topological groupoid $\GG$
on a topological space $X$ along a (continuous) map $\epsilon:X\to \GG_0$
is a map $\mu:X\times_{\GG_0}\GG_1\to X$, $(x,g)\mapsto xg$,
which is defined for each pair
$(x,g)\in X\times\GG_1$ with $\epsilon(x)=\trg(g)$ and satisfies the usual
properties of an action.
For such a right $\GG$-action we write
$x\GG=\mu({x},\trg^{-1}(x))$ for the {\em orbit} through $x$ and
$X/\GG$ for the associated {\em space of orbits}.
We have the associated {\em translation} groupoid
$X\rtimes\GG=(X\times_{\GG_0}\GG_1 \rightrightarrows X)$,
which is the topological groupoid in which
the source map is the action $\mu$, the target map is the first projection and
the product is given by $(x,g)(x',g')=(x,gg')$.
Analogously, one defines a left action of a topological groupoid
$\HH$ on $X$ and the associated translation groupoid $\HH\ltimes X$.

\subsection{Morita category of topological groupoids}
Let $\GG$ and $\HH$ be topological groupoids.
A {\em principal $\GG$-bundle} over $\HH$ is a space $P$, equipped with a left
$\HH$-action along $\pi:P\to \HH_0$ and
a right $\GG$-action along $\epsilon:P\to \GG_0$
such that
\begin{enumerate}
\item [(i)]   the map $\pi$ has local sections (\textit{i.e.}\ for any
              $y\in\HH_0$ there exist a section of $\pi$, defined on
              an open neighbourhood of $y$),
\item [(ii)]  for every $h\in\HH_1$, $p\in P$ and $g\in\GG_1$ with
              $s(h)=\pi(p)$ and $\epsilon(p)=t(g)$ we have
              $\pi(pg)=\pi(p)$, $\epsilon(hp)=\epsilon(p)$ and
              $h(pg)=(hp)g$, and
\item [(iii)] the map $P\times_{\GG_0}\GG_1\to P\times_{\HH_0} P$, $(p,g)\mapsto(p,pg)$,
              is a homeomorphism.
\end{enumerate}
A {\em morphism} between principal $\GG$-bundles over $\HH$
is an $\HH$-$\GG$-equivariant map between them. Any such homomorphism
is in fact an isomorphism.

For a principal $\GG$-bundle $P$ over $\HH$ and a principal $\HH$-bundle $Q$
over $\KK$ we have the tensor product $Q\otimes P=Q\otimes_{\HH} P$, which
is a principal $\GG$-bundle over $\KK$.
The space $Q\otimes P$ is the space of
orbits of the diagonal $\HH$-action on the fibered product
$Q\times_{\HH_0}P$ \cite{MoerdijkMrcun2005,Mrcun1996}.
The Morita category $\GPD$ of topological
groupoids is the category with topological groupoids as objects and
isomorphism classes of principal bundles as morphisms,
and with composition induced by the tensor product.
Morita equivalences are the isomorphisms in $\GPD$.
The category of topological spaces
is a full subcategory of the Morita category $\GPD$.
There is a natural functor $\langle\oo\rangle:\Gpd\to\GPD$,
which is identity on objects, while it maps
a homomorphism $\phi:\HH\to\GG$ to the (isomorphism class of the)
principal $\GG$-bundle 
$\langle\phi\rangle=\HH_0\times_{\GG_0}\GG_1$ over $\HH$.
We also have the Morita bicategory of topological groupoids
in which principal bundles
are $1$-morphisms and morphisms of principal bundles are $2$-morphisms.

\subsection{Pointed and marked topological groupoids}
A {\em pointed} topological groupoid is a topological groupoid $\GG$,
together with a {\em basepoint} $a\in\GG_0$ (see \cite{Haefliger2010,Mrcun1996}).
More generally, a {\em marked} topological groupoid is a pair $(\GG,A)$, where
$\GG$ is a topological groupoid and $A$ is a subset of $\GG_0$.
A pointed topological groupoid $(\GG,a)$ can therefore be seen
as the marked topological groupoid $(\GG,\{a\})$.
A {\em homomorphism} between marked topological groupoids
$\phi:(\HH,B)\to(\GG,A)$ is a homomorphism $\phi:\HH\to\GG$
such that $\phi(B)\subset A$. A principal $(\GG,A)$-bundle
over $(\HH,B)$ is a pair $(P,\sigma)$, where
$P$ is a principal $\GG$-bundle over $\HH$ and $\sigma:B\to P$
is a section of $\pi:P\to\HH_0$
such that $\epsilon:(P,\sigma(B))\to (\GG_0,A)$ is a map of
topological pairs and $\pi:(P,\sigma(B))\to(\HH_0,B)$ is a map
of topological pairs with local sections (\textit{i.e.}\ for any
$b\in B$ there exist a section of $\pi$, defined on an open
neighbourhood $V$ of $b$ in $\HH_0$, which is a map
of topological pairs $(V,V\cap B)\to (P,\sigma(B))$).
In particular, if $(\GG,a)$ and $(\HH,b)$
are pointed topological groupoids and
$(P,\sigma)$ is principal $(\GG,a)$-bundle
over $(\HH,b)$, then we write
$(P,\sigma)=(P,\sigma(b))$.
Morphisms between principal $(\GG,A)$-bundles
over $(\HH,B)$ are morphisms of principal $\GG$-bundles
over $\HH$ which commute the given sections.
One can check that the tensor product of principal bundles
with sections has a natural section as well, thus we obtain
the Morita category $\GPD_{[\ast]}$ of marked topological groupoids.
As a full subcategory of $\GPD_{[\ast]}$, we have the Morita
category $\GPD_{\ast}$ of pointed topological groupoids.

\subsection{Group-marked topological groupoids}
Let $K$ be a fixed locally compact Hausdorff topological group.
A {\em $K$-basegroup} in a topological groupoid $\GG$ is
a homomorphism $\varphi:K\to\GG$, where we view $K$
as the topological groupoid $(K\rightrightarrows\ast)$.
In particular, such a
$K$-basegroup is
a homomorphism of topological groups $\varphi:K\to\GG_{\varphi(\ast)}$.
We write
$$ \hom(K,\GG)=\Gpd((K\rightrightarrows\ast),\GG) $$
for the space of all $K$-basegroups of $\GG$, equipped with
the compact-open topology, and $\omega:\hom(K,\GG)\to\GG_0$
for the natural map given by $\omega(\varphi)=\varphi(\ast)$.
Any homomorphism $\phi\in\Gpd(\HH,\GG)$ induces a map
$\phi_\ast:\hom(K,\HH)\to\hom(K,\GG)$.
A {\em $K$-pointed} topological groupoid is a topological
groupoid $\GG$ together with a $K$-basegroup $\rho$ in $\GG$;
we denote such a $K$-pointed topological groupoid by
$(\GG,a,\rho)$, where $a=\rho(\ast)$.
A homomorphism between $K$-pointed topological groupoids
$(\HH,b,\upsilon)$ and $(\GG,a,\rho)$ is a homomorphism
$\phi\in\Gpd(\HH,\GG)$ with $\phi_\ast(\upsilon)=\rho$.
In this way we obtain the category $\Gpd_K$ of $K$-pointed
topological groupoids, which is in other words the coslice
category $(K\rightrightarrows\ast)\downarrow\Gpd$.

For a principal $\GG$-bundle $P$ over $\HH$, any point
$p\in P$ induces a homomorphism of groups
$P_p:\HH_{\pi(p)}\to\GG_{\epsilon(p)}$,
uniquely determined by the property that
$h p=p P_p(h)$ for any $h\in \HH_{\pi(p)}$.
A {\em principal $(\GG,a,\rho)$-bundle over} $(\HH,b,\upsilon)$
is a principal
$(\GG,a)$-bundle $(P,p)$ over $(\HH,b)$
satisfying $P_p\com\upsilon=\rho$.
We obtain the Morita category $\GPD_K$ of $K$-pointed
topological groupoids, where the morphisms are given
by the isomorphism classes of principal bundles.

More generally, a {\em $K$-marked} topological groupoid
is a triple $(\GG,A,\rho)$, where $A\subset\GG_0$ and
$\rho:K\times A\to\GG$ is a continuous functor which
is identity on objects.
A {\em principal $(\GG,A,\rho)$-bundle over}
$(\HH,B,\upsilon)$  is a principal
$(\GG,A)$-bundle $(P,\sigma)$ over $(\HH,B)$ satisfying
$P_{\sigma(b)}(\upsilon(k,b))=\rho(k,\epsilon(\sigma(b)))$
for any $b\in B$ and $k\in K$. Again we obtain the Morita category
$\GPD_{[K]}$ of $K$-marked
topological groupoids, where the morphisms are given
by isomorphism classes of principal bundles.
The category $\GPD_K$ is a full subcategory of  $\GPD_{[K]}$.
An example of a $K$-marked topological groupoid
is $(K\times I^n,\partial I^n,\tau)$, where
$\tau:K\times\partial I^n\to K\times I^n$ is the inclusion.

\subsection{Homotopy}
Let $(\GG,A,\rho),(\HH,B,\upsilon)$ be $K$-marked
topological groupoids and $(P,\sigma),(P',\sigma')$
principal $(\GG,A,\rho)$-bundles over $(\HH,B,\upsilon)$.
A {\em homotopy} from $(P,\sigma)$ to $(P',\sigma')$
is principal $(\GG,A,\rho)$-bundle $(H,\lambda)$ over
$(\HH\times I,B\times I,\upsilon\times I)$ such
that the restrictions
of $(H,\lambda)$ to
$(\HH\times \{0\},B\times \{0\},\upsilon\times \{0\})$
and
$(\HH\times \{1\},B\times \{1\},\upsilon\times \{1\})$
are isomorphic to
$(P,\sigma)$ respectively $(P',\sigma')$.
We say that $(P,\sigma),(P',\sigma')$ are {\em homotopic}
if such a homotopy exists. One may show that homotopy is
an equivalence relation on the set of Morita morphisms
from $(\HH,B,\upsilon)$ to $(\GG,A,\rho)$
and that this equivalence relation behaves well
with respect to the composition in $\GPD_{[K]}$.
(To show the transitivity of the relation, one needs to
concatenate two homotopies, and this should be done by inserting
a trivial homotopy in the middle to ensure that the concatenated
principal bundle is locally trivial.)
We denote by
$$ \hGPD_{[K]}((\HH,B,\upsilon),(\GG,A,\rho))
=[(\HH,B,\upsilon),(\GG,A,\rho)] $$
the set of homotopy classes of principal
$(\GG,A,\rho)$-bundles over $(\HH,B,\upsilon)$.
We obtain the {\em homotopy Morita category}
$\hGPD_{[K]}$ of $K$-marked topological groupoids.

\subsection{Homotopy groups}
We define the {\em $n$-th homotopy group}
of a $K$-pointed topological groupoid $(\GG,a,\rho)$
by
$$ \pi_n^K(\GG)=\pi_n^K(\GG,a,\rho)= [(K\times I^n,\partial I^n,\tau),(\GG,a,\rho)]\;.$$
As usual, this is just a pointed set for $n=0$, a group for
$n=1$ and an abelian group for $n\geq 2$. For $K=1$ we recover the usual
homotopy groups $\pi_n(\GG,a)$ of the pointed topological groupoid
$(\GG,a)$. By definition we see that $\pi_n^K$ is homotopy Morita invariant,
as it is a functor defined on the category $\GPD_{K}$ and also on
$\hGPD_{K}$. Any homomorphism 
$K\to K'$ between locally compact Hausdorff topological groups
induces a natural transformation from $\pi_n^{K'}$ and
$\pi_n^K$ over the natural functor $\GPD_{K'}\to\GPD_K$.

\section{Homotopy sequence of a topological groupoid with a basegroup}

\subsection{The groupoid $\GG^K$}
Let $K$ be a locally compact Hausdorff topological group.
For a topological groupoid $\GG$ we have
the natural map
$\omega:\hom(K,\GG)\to\GG_0$, $\varphi\mapsto\varphi(\ast)$.
Now observe that there is a natural right (adjoint) $\GG$-action
on $\hom(K,\GG)$ along $\omega$, given by
$$ (\varphi g)(k) = g^{-1}\varphi(k)g = (\Ad_{g^{-1}}\com\varphi)(k)$$
for any $k\in K$. We denote by
$$ \GG^K  = \hom(K,\GG)\rtimes\GG $$
the associated translation groupoid. In particular, the map
$\omega$ induces a homomorphism
$\omega=\omega_\GG : \GG^K\to\GG$.
The definition of $\GG^K$ can be extended to a functor.
Indeed, if $P$ is a principal
$\GG$-bundle over $\HH$, we take
$$ P^K=\hom(K,\HH)\times_{\HH_0} P \;.$$
The space $P^K$ has a natural structure of
a principal $\GG^K$-bundle over $\HH^K$, given by maps
$\pi(\varphi,p)=\varphi$, $\epsilon(\varphi,p)=P_p\com\varphi$ and
actions
$h(\varphi,p)=(\Ad_h\com\varphi,hp)$,
$(\varphi,p)g=(\varphi,pg)$.
One can check that this gives us a functor
$(\oo)^K:\GPD\to\GPD$. Note that we have a natural isomorphism
between $P^K\otimes\langle\omega_\GG\rangle$ and
$\langle\omega_\HH\rangle\otimes P$.
Note also that a homomorphism $\phi:\HH\to\GG$ induces
a homomorphism $\phi^K:\HH^K\to\GG^K$ in the obvious way and
that $\langle\phi^K\rangle$ is naturally isomorphic 
to $\langle\phi\rangle^K$.

\begin{prop}\label{exponent}
Let $\GG,\HH$ be topological groupoids and $K$
a locally compact Hausdorff topological group.
There is a natural bijection
$$ \GPD(K\times\HH,\GG)\cong\GPD(\HH,\GG^K)\;. $$
\end{prop}
\begin{proof}
Let $P$ be a principal $\GG$-bundle over $K\times\HH$.
For any $p\in P$ we have the associated
homomorphism of topological groups
$P_p:K\times\HH_{\pi(p)}\to\GG_{\epsilon(p)}$.
There is also the associated map $\bar{\epsilon}:P\to\hom(K,\GG)$
given by $\bar{\epsilon}(p)(k)=P_p(k,1_{\pi(p)})$. Explicitly,
the map $\bar{\epsilon}$ is characterized by the equality
\begin{equation}\label{barepsilon}
(k,1_{\pi(p)})p=p\bar{\epsilon}(p)(k)\;,
\end{equation}
which holds for any $k\in K$.
One can check that $\bar{\epsilon}$ is $\GG$-equivariant,
so we have a right $\GG^K$-action on $P$ along $\bar{\epsilon}$.
Furthermore, this action commutes with the left $\HH$-action
and $P$ is also a principal $\GG^K$-bundle over $\HH$.
Finally, observe that $\epsilon$ can be reconstructed from
$\bar{\epsilon}$ as $\epsilon=\omega\com\bar{\epsilon}$,
and the left $K$-action on $P$ can be reconstructed
by Equation (\ref{barepsilon}).
\end{proof}

\Rem
In fact we proved that we have a natural
isomorphism of groupoids
$$ \underline{\GPD}(K\times\HH,\GG)\cong\underline{\GPD}(\HH,\GG^K)\;, $$
where we use the underlined notation $\underline{\GPD}$ to denote
the set of principal bundles rather than the set of their isomorphism
classes, \textit{i.e.}\ the 1-morphisms in the Morita bicategory. 
We can lift this to the context of $K$-marked groupoids. In particular, we
obtain a natural isomorphism of groupoids
$$ \underline{\GPD}_{[K]}((K\times I^n,\partial I^n,\tau),(\GG,a,\rho))
   \cong\underline{\GPD}_{[\ast]}((I^n,\partial I^n),(\GG^K,\rho)) \;, $$
for any $K$-pointed topological groupoid
$(\GG,a,\rho)$. 

The proposition shows that the groupoid $\GG^K$ behaves
as the exponent in the Morita category.
For a more general study of exponents in the context of
topological stacks, see \cite{Carchedi}.

\begin{cor}\label{pin}
Let $K$ be a locally compact Hausdorff topological group.
For any $K$-pointed topological groupoid
$(\GG,a,\rho)$ there is a natural isomorphism
$$ \pi_n^K(\GG,a,\rho)\cong\pi_n(\GG^K,\rho) \;.$$
\end{cor}

\subsection{Serre $K$-fibrations}
Let $\HH$ and $\GG$ be topological groupoids and let
$P$ be a principal $\GG$-bundle over $\HH$. Recall from \cite{Jelenc}
that $P$ is a {\em Serre fibration} if
for any given principal $\HH$-bundle $\alpha$ over $I^n$,
principal $\GG$-bundle $\beta$ over $I^{n+1}$ and
isomorphism $\chi$ from
$\alpha\otimes P$ to $\beta|_{I^n\times\left\{0\right\}}$,
there exist
a principal $\HH$-bundle $\bar{\alpha}$ over $I^{n+1}$,
a principal $\GG$-bundle $\bar{\beta}$ over $I^{n+1}$ and
isomorphisms of principal bundles
$\bar{\chi}:\bar{\alpha}\otimes P\to \bar{\beta}$,
$\zeta:\alpha\to\bar{\alpha}|_{I^n\times\{0\}}$ and
$\xi:\beta\to\bar{\beta}$ such that
$\bar{\chi}\com (\zeta\otimes P)=\xi\com\chi$.
This notion is defined in the Morita bicategory of topological groupoids,
however it is clearly a well-defined property
of a Morita map in $\GPD$.
It is a generalization
of the notion of Serre fibration in the category of topological spaces,
and can be stated as the existence of the
morphism $\bar{\alpha}$ in the commutative diagram
$$
\xymatrix{
I^n\ar[r]^-\alpha\ar[d]& \HH\ar[d]^-{P}\\
I^{n+1}\ar[r]_-\beta\ar@{.>}[ru]^-{\bar{\alpha}} & \GG
}
$$
of 1-arrows in the Morita bicategory of topological groupoids.
We say that a homomorphism $\phi:\HH\to\GG$ of topological
groupoids is a Serre fibration if the associated
principal $\GG$-bundle $\langle\phi\rangle$ over $\HH$
is a Serre fibration.

Let us state some basic properties of Serre fibrations which we will
use in the rest of the paper (see \cite{Jelenc}, compare to \cite{Noohi}):

(i) 
Any Morita equivalence (\textit{i.e.}\ an isomorphism in $\GPD$)
is a Serre fibration.

(ii)
The composition of Serre fibrations in $\GPD$ is
a Serre fibration.

(iii)
If $\phi:\HH\to\GG$ is a homomorphism between topological groupoids
such that $\phi:\HH_0\to\GG_0$ is a Serre fibration and
$(\trg,\phi):\HH_1\to\HH_0\times_{\GG_0}\GG_1$
is a surjective Serre fibration, then $\phi$ is a Serre fibration.

(iv)
Any principal $(\GG,a)$-bundle 
$(P,p)$ over $(\HH,b)$, which is a Serre fibration,
induces a natural long exact sequence
$$
\ldots\to\pi_n(\HH\ltimes\epsilon^{-1}(a))\stackrel{\pi_n(\textnormal{pr}_1)}{\to}\pi_n(\HH)
\stackrel{\pi_{n}(P)}{\to}\pi_n(\GG)\to\pi_{n-1}(\HH\ltimes\epsilon^{-1}(a))\to\ldots 
$$ 
The translation groupoid $\HH\ltimes\epsilon^{-1}(a)$ is called
the {\em fiber} of $P$ over $a$.

We say that a principal $\GG$-bundle $P$ over $\HH$ is
a {\em Serre $K$-fibration} if the associated principal
$\GG^K$-bundle $P^K$ over $\HH^K$ is a Serre fibration.
By Proposition \ref{exponent} we see that this is true if and only if
for any given principal $\HH$-bundle $\alpha$ over $K\times I^n$,
principal $\GG$-bundle $\beta$ over $K\times I^{n+1}$ and
isomorphism $\chi$ from
$\alpha\otimes P$ to $\beta|_{K\times I^n\times\left\{0\right\}}$,
there exist
a principal $\HH$-bundle $\bar{\alpha}$ over $K\times I^{n+1}$,
a principal $\GG$-bundle $\bar{\beta}$ over $K\times I^{n+1}$ and
isomorphisms of principal bundles
$\bar{\chi}:\bar{\alpha}\otimes P\to \bar{\beta}$,
$\zeta:\alpha\to\bar{\alpha}|_{K\times I^n\times\{0\}}$ and
$\xi:\beta\to\bar{\beta}$ such that
$\bar{\chi}\com (\zeta\otimes P)=\xi\com\chi$.
$$
\xymatrix{
K\times I^n\ar[r]^-\alpha\ar[d]& \HH\ar[d]^-{P}\\
K\times I^{n+1}\ar[r]_-\beta\ar@{.>}[ru]^-{\bar{\alpha}} & \GG
}
$$
Analogously, we say that 
a homomorphism $\phi:\HH\to\GG$ is a Serre $K$-fibration
if $\phi^K:\HH^K\to\GG^K$ is a Serre fibration, which is true
of course if and only if $\langle\phi\rangle^K$ is a Serre fibration.

\begin{lem}\label{lem}
Let $\phi:\HH\to\GG$ be a homomorphism of topological groupoids
such that the map
$(\trg,\phi):\HH_1\to\HH_0\times_{\GG_0}\GG_1$
is a surjective Serre fibration.
Suppose that $X$ is a space with a right $\GG$-action,
$Y$ a space with a right $\HH$-action
and $f:Y\to X$ a $\phi$-equivariant Serre fibration.
Then the induced homomorphism
$Y\rtimes\HH\to X\rtimes\GG$ is a Serre fibration.
\end{lem}
\begin{proof}
We have to show that the map
$$ (\trg,f\times\phi):(Y\rtimes\HH)_1\to
   (Y\rtimes\HH)_0\times_{(X\rtimes\GG)_0}(X\rtimes\GG)_1$$
is a surjective Serre fibration. Equivalently, we have to show
that the map 
$Y\times_{\HH_0}\HH_1\to Y\times_{\HH_0} (\HH_0\times_{\GG_0}\GG_1)$,
$(y,h)\mapsto (y,\trg(h),\phi(h))$, is a surjective
Serre fibration, which is true because it is a pull-back of
the surjective Serre fibration
$(\trg,\phi):\HH_1\to\HH_0\times_{\GG_0}\GG_1$.
\end{proof}

\begin{prop}
Let $K$ be a locally compact
Hausdorff topological group and
let $\phi:\HH\to\GG$ be a homomorphism of topological groupoids
such that the map
$\phi_\ast:\hom(K,\HH)\to\hom(K,\GG)$ is a Serre fibration and
the map
$(\trg,\phi):\HH_1\to\HH_0\times_{\GG_0}\GG_1$
is a surjective Serre fibration. Then $\phi$ is
a Serre $K$-fibration.
\end{prop}
\begin{proof}
Note that the assumption that
$\phi_\ast:\hom(K,\HH)\to\hom(K,\GG)$ is a Serre
fibration implies that $\phi:\HH_0\to\GG_0$ is also
a Serre fibration. The proposition is hence
a consequence of Lemma \ref{lem}.
\end{proof}

\begin{prop}\label{lexPK}
Let $K$ be a locally compact Hausdorff topological group,
let $(\GG,a,\rho)$, $(\HH,b,\upsilon)$ be $K$-pointed
topological groupoids
and let
$(P,p)$ be a principal $(\GG,a,\rho)$-bundle over $(\HH,b,\upsilon)$
which is a Serre fibration.
We have a natural long exact sequence
\begin{equation*}
\begin{split}
\ldots
\to&\pi_n(\HH\ltimes(\hom(K,\HH)\times\epsilon^{-1}(a)),(\rho,p))
\to\pi_n^K(\HH,b,\upsilon)
\stackrel{\pi_{n}(P)}{\to}\pi_n^K(\GG,a,\rho)\\
\to&\pi_{n-1}(\HH\ltimes(\hom(K,\HH)\times\epsilon^{-1}(a)),(\rho,p))
\to
\ldots 
\end{split}
\end{equation*}
\end{prop}

\begin{proof}
This is simply the long exact sequence of homotopy groups
associated to the Serre fibration $P^K$.
\end{proof}

\begin{prop}
Let $\GG$ be a topological groupoid and $K$ a locally compact
Hausdorff topological group. If $\omega:\hom(K,\GG)\to\GG_0$
is a Serre fibration, then $\omega:\GG^K\to\GG$ is a Serre fibration.
\end{prop}

\begin{proof}
This is a direct consequence of Lemma \ref{lem}.
\end{proof}

\begin{theo}\label{lex}
Let $K$ be a locally compact Hausdorff topological group
and $(\GG,a,\rho)$ a $K$-pointed topological groupoid such that
$\omega:\GG^K\to\GG$ is a Serre fibration.
We have a natural long exact sequence
$$
\ldots\to\pi_n(\hom(K,\GG_a))\to\pi_n^K(\GG)
\stackrel{\pi_{n}(\omega)}{\to}\pi_n(\GG)
\stackrel{\partial}{\to}\pi_{n-1}(\hom(K,\GG_a))\to\ldots 
$$
\end{theo}

\begin{proof}
The long exact sequence is the one associated to
the Serre fibration $\omega:\GG^K\to\GG$.
This follows from Corollary \ref{pin} and
the fact that the fiber of $\langle\omega\rangle$ over $a$
is Morita equivalent to the space $\hom(K,\GG_a)$. Indeed,
we have
$$ \langle\omega\rangle=\hom(K,\GG)\times_{\GG_0}\GG_1\;,$$
so the fiber of $\langle\omega\rangle$ over $a$ is the
translation groupoid associated to the $\GG^K$-action on the
space
$$ F=\hom(K,\GG)\times_{\GG_0}\GG(a,\oo)\;,$$
where $\GG(a,\oo)=\src^{-1}(a)$ is the source-fiber over $a$ in $\GG$.
Now note that the
$\GG^K$-action on $F$ is free and the subspace
$\hom(K,\GG)\times_{\GG_0}\{1_a\}\cong \hom(K,\GG_a)$
of $F$ intersects each orbit
of $\GG^K$-action in exactly one point.
One can then see that the natural inclusion of the space
$\hom(K,\GG_a)$ into the fiber $\GG^K\ltimes F$
gives a Morita equivalence.
\end{proof}

It turns out that in concrete examples
it is often too optimistic to expect that
$\omega:\GG^K\to\GG$ is a Serre fibration,
but it may well happen that this homomorphism is a Serre fibration
when restricted to a path-component of the topological
groupoid $\GG^K$. (We say that a topological groupoid $\HH$ is
{\em path-connected} if $\pi_0(\HH)=1$. A path-component of
a topological groupoid $\HH$ is a maximal path-connected
(full) subgroupoid of $\HH$. The space of objects of
a path-component of $\HH$ is a minimal $\HH$-invariant union
of path-components of the space $\HH_0$.)
Therefore, let us introduce the following notation:
For a $K$-pointed path-connected topological groupoid
$(\GG,a,\rho)$, denote by $\GG^K_\rho$ the path-component of $\GG^K$
with $\rho\in(\GG^K_\rho)_0$, and write
$\hom(K,\GG)_\rho=(\GG^K_\rho)_0$. 
Denote also $\hom(K,\GG)_{a,\rho}=\hom(K,\GG_a)\cap\hom(K,\GG)_\rho$.
Lemma \ref{lem} now implies:

\begin{prop}
Let $K$ be a locally compact Hausdorff topological group and 
$(\GG,a,\rho)$ a $K$-pointed path-connected topological groupoid.
If $\omega:\hom(K,\GG)_\rho\to\GG_0$
is a Serre fibration, then $\omega:\GG^K_\rho\to\GG$ is a Serre fibration.
\end{prop}

\begin{theo}\label{lexcon}
Let $K$ be a locally compact Hausdorff topological group
and $(\GG,a,\rho)$ a $K$-pointed path-connected topological groupoid such
that $\omega:\GG^K_\rho\to\GG$ is a Serre fibration.
We have a natural long exact sequence
\begin{equation*}
\begin{split}
\ldots\to\pi_n(\hom(K,\GG_a))\to\pi_n^K(\GG)
\stackrel{\pi_{n}(\omega)}{\to}&\pi_n(\GG)
\stackrel{\partial}{\to}\pi_{n-1}(\hom(K,\GG_a))\to\ldots \\
\ldots\to\pi_1(\hom(K,\GG_a))\to\pi_1^K(\GG)
\stackrel{\pi_{1}(\omega)}{\to}&\pi_1(\GG)
\stackrel{\partial}{\to}\pi_{0}(\hom(K,\GG)_{a,\rho})\to 1\;.
\end{split}
\end{equation*}
\end{theo}

\begin{proof}
The proof is analogous to the proof of Theorem \ref{lex},
with $\GG^K$ replaced by $\GG^K_\rho$ and
$\hom(K,\GG_a)$ replaced by $\hom(K,\GG)_{a,\rho}$.
We use that $\pi_0(\GG^K_\rho)=1$,
$\pi_n(\GG^K)=\pi_n(\GG^K_\rho)$ and
$\pi_n(\hom(K,\GG)_{a,\rho})=\pi_n(\hom(K,\GG_a))$
for any $n\geq 1$.
\end{proof}

\section{Obstruction to presentability of proper regular Lie groupoids}

\subsection{Normal representation of a Lie groupoid}
Recall that a {Lie groupoid} is a topological groupoid $\GG$
such that both $\GG_0$ and $\GG_1$ are smooth manifolds
and all the structure maps are smooth, with the source map being a submersion.
We also assume that $\GG_0$ and the fibers of the source map are
Hausdorff, but in general $\GG_1$ may not be Hausdorff.
Of course, we consider the smooth homomorphisms between Lie groupoids,
smooth actions of Lie groupoids
and smooth principal bundles. If $\GG$ and $\HH$ are Lie groupoids, then a smooth
principal $\GG$-bundle over $\HH$ is a principal $\GG$-bundle $P$ over $\HH$
such that $P$ is a smooth manifold, all the structure maps of $P$ are smooth,
the projection $P\to\HH_0$ is a surjective submersion and
the map $P\times_{\GG_0}\GG_1\to P\times_{\HH_0} P$, $(p,g)\mapsto(p,pg)$, is
a diffeomorphism. The (smooth) isomorphism classes of such smooth principal bundles
are morphisms in the (smooth) Morita
category of Lie groupoids (for details and examples,
see \cite{MoerdijkMrcun2003,MoerdijkMrcun2005}).
Smooth manifolds and Lie groups are examples of Lie groupoids, as are
the translation groupoids of smooth actions of Lie groupoids.
The isotropy groups of a Lie groupoid are Lie groups.

Recall that a Lie groupoid $\GG$ is {\em proper} if it is Hausdorff
and the map $(\src,\trg):\GG_1\to\GG_0\times\GG_0$ is proper.
A Lie groupoid $\GG$ is {\em regular} if all its isotropy groups have
the same dimension. A {\em foliation} Lie groupoid is a Lie groupoid
with discrete isotropy groups. All these notions are stable under  (smooth)
Morita equivalence
(see \cite{MoerdijkMrcun2003,MoerdijkMrcun2005,Weinstein}).
A Lie groupoid $\GG$ is {\em \'{e}tale} if its source map is a local diffeomorphism.
A Lie groupoid is a foliation groupoid if and only if it is
Morita equivalent to an \'{e}tale Lie groupoid.
Orbifolds can be viewed as proper foliation groupoids \cite{Moerdijk2002}.

Let $\GG$ be a Lie groupoid. For any $x\in \GG_0$
we denote by $\cO_x(\GG)=\GG x$ the orbit of $\GG$ through $x$.
The orbit is an immersed submanifold of $\GG_0$, with the smooth
structure given
by the principal $\GG_x$-bundle $\trg:\GG(x,\oo)\to\cO_x(\GG)$.
The normal space of $\GG$ at $x$ is the quotient
$\Nor_x\GG=\Tan_x \GG_0/\Tan_x\cO_x(\GG)$ of the tangent space of $\GG_0$ and
the tangent space of the orbit at $x$.
The union $\Nor\GG$ of all normal spaces $\Nor_x\GG$, $x\in\GG_0$, is
a topological family of vector spaces which is locally trivial
along an orbit, but not locally over $\GG_0$ in general, as dimension of
$\Nor_x\GG$ may vary. 
There is a continuous linear left $\GG$-action $\nu$ on $\Nor_x\GG$ along the natural
projection to $\GG_0$. For any $g\in\GG(x,x')$, the linear isomorphism
$$ \nu(g,\oo):\Nor_x\GG\to \Nor_{x'}\GG $$
is determined by the equality
$\nu(g,[(d_g\src)(v)])=[(d_g\trg)(v)]$ for any
$v\in\Tan_g\GG_1$ (for details, see \cite{EvensLuWeinstein}).
This action restricts to a Lie group representation
$$ \nu_x:\GG_x\to \GL(\Nor_x\GG) \;,$$
and the kernel
$$ \K_x(\GG)=\Ker(\nu_x) $$
is a closed Lie subgroup of $\GG_x$.
The union $\K(\GG)$ of $\K_x(\GG)$, $x\in \GG_0$, is a topological
family of Lie groups, a subfamily of the union $\I(\GG)$
of isotropy groups of $\GG$,
which is also a topological family of Lie groups.
The elements of $\K(\GG)$ are called the {\em ineffective} arrows of $\GG$.
Note that $\K(\GG)$ is normal in $\GG$, \textit{i.e.}\ for any
$g\in\GG(x,x')$ we have $\Ad_g(\K_x(\GG))=\K_{x'}(\GG)$, so
we have a left (adjoint) $\GG$-action on $\K(\GG)$ along the projection to $\GG_0$.
This follows from the equation
$\Ad_{\nu(g,\oo)}\com\nu_x=\nu_{x'}\com\Ad_g$.

It is important to note that the normal family $\Nor\GG$, the
normal action $\nu$, the isotropy family $\I(\GG)$ and
the ineffective isotropy family $\K(\GG)$
all behave well with respect to (smooth) Morita morphisms.
More precisely, let $P$ be a (smooth) principal $\GG$-bundle over a Lie groupoid $\HH$.
Any $p\in P$ defines a linear map
$d_pP:\Nor_{\pi(p)}\HH\to\Nor_{\epsilon(p)}\GG$, determined by
the equality $d_pP([(d_p\pi)(v)])=[(d_p\epsilon)(v)]$ for any
$v\in\Tan_pP$. One can check that
$\nu_{\epsilon(p)}(P_p(h))\com  d_pP
= d_pP\com\nu_{\pi(p)}(h)$ for any $h\in \HH_{\pi(p)}$.
If $P$ is a Morita equivalence, the maps
$P_p$ and $d_pP$ are isomorphisms and
$P_p(\K_{\pi(p)}(\HH))=\K_{\epsilon(p)}(\GG)$.

If $\GG$ is regular, then the connected
components of the orbits of $\GG$ define a regular foliation $\cF(\GG)$
on $M$, and $\Nor\GG$ is a (locally trivial) smooth vector bundle on $\GG_0$ 
equal to the normal bundle of the foliation $\cF(\GG)$.
The action $\nu$ of $\GG$ on $\Nor\GG$ is then a (smooth) representation
of the Lie groupoid $\GG$.

If $\GG$ is regular and proper, then all the isotropy groups
and all the ineffective isotropy groups are compact. However,
this does not imply that the isotropy bundle $\I(\GG)$ is a locally trivial
bundle of Lie groups. It turns out, on the other hand, that
$\K(\GG)$ is, at least if $\GG$ is connected. In a different way,
this was proved in \cite{Moerdijk2003} for the bundle $\K(\GG)^{\circ}$, which
is the component of the trivial section in the bundle of $\K(\GG)$.

\begin{prop}
If $\GG$ is a connected proper regular Lie groupoid, then the ineffective isotropy
bundle $\K(\GG)$ is a locally trivial bundle of Lie groups.
\end{prop}

\begin{proof}
This is a direct consequence of the local
linearizability of proper effective
Lie groupoids \cite{CrainicStruchiner,Weinstein}.
\end{proof}

\subsection{Homotopy sequence of a proper regular Lie groupoid}

\begin{prop}\label{propprg}
Let $K$ be a compact Lie group,
let $\GG$ be a connected proper regular Lie groupoid
with a basepoint $a$ and let $\rho$ be a $K$-basegroup
in $\GG$ with image in $\K_a(\GG)$.

(i)
Any element $\varphi\in \hom(K,\GG)_\rho$ has its image
in $\K(\GG)$. If $\rho:K\to\K_a(\GG)$ is an isomorphism,
then $\varphi:K\to\K_{\varphi(\ast)}(\GG)$ is also an isomorphism.

(ii)
The map $\omega:\hom(K,\GG)_\rho\to \GG_0$  and the homomorphism
$\omega:\GG^K_\rho\to \GG$ are Serre fibrations. 
\end{prop}

\begin{proof}
(i) 
Recall that the space $\hom(K,\GG)_\rho$ is a minimal $\GG$-invariant union
of path-components of the space $\hom(K,\GG)$. The bundle
$\K(\GG)$ is normal in $\GG$, and therefore it is sufficient to show
that any $\varphi$, which is in the same path-component of the space
$\hom(K,\GG)$ as $\rho$, has in fact image in $\K(\GG)$.
In other words, if $\gamma$ is a path in $\hom(K,\GG)$ starting
at $\rho$, we have to show that the image of $\gamma(t)$
is in $\K(\GG)$ for any $t\in [0,1]$.
The path $\gamma$ gives us a path $\omega\com\gamma$ in $\GG_0$.
Since $\GG$ is proper and hence locally linearizable
\cite{CrainicStruchiner,Weinstein}, we may cover
the image of $\omega\com\gamma$ with finitely many open subsets
of $\GG_0$ on which the groupoid $\GG$ is linearizable.
We may therefore assume, without loss of generality, that
the image of $\omega\com\gamma$ lies in an open subset $U$
of $\GG_0$ such that the restriction of $\GG$ to $U$ is
linearizable. Explicitly, this means that $\GG|_U$ is isomorphic to
the restriction of the translation groupoid
$$ \cN_{\cO}(\GG)=\GG|_{\cO_a}\ltimes(\src^{-1}(a)\times_{\GG_a}\Nor_a(\GG)) $$
to an open subset $V$ of $\src^{-1}(a)\times_{\GG_a}\Nor_a(\GG)$,
where $\GG|_{\cO_a}$ is the transitive groupoid obtained as
the restriction of $\GG$ to the orbit $\cO_a$ of $\GG$
through $a$. The isomorphism $\phi$ from $\GG|_U$ to
$\cN_{\cO}(\GG)|_{V}$ is the canonical one when restricted
to the orbit $\cO_a$.

The quotient map
$q:\src^{-1}(a)\times\Nor_a(\GG)\to\src^{-1}(a)\times_{\GG_a}\Nor_a(\GG)$
is a principal $\GG_a$-bundle,
so we can lift the path $\phi\com\omega\com\gamma$ to a path
$\tilde{\gamma}$ in $\src^{-1}(a)\times\Nor_a(\GG)$.
We write $\tilde{\gamma}(t)=(g_t,v_t)$
and put $\alpha(t)=\Ad_{g_{t}^{-1}}\com\phi_\ast(\gamma(t))$.
Note that $\alpha(t)$ is a homomorphism from $K$ to
the isotropy group of $\cN_{\cO}(\GG)$ at $q(1_a,v_t)$,
which is a subgroup of $\GG_a$.
Since $\alpha(t)$ depends continuously on $t$,
there exist $h_t\in\GG_a$ such that
$\alpha(t)=\Ad_{h_t}\com\alpha(0)$ 
(see \cite[Lemma 38.1]{ConnerFloyd}).
It follows that
$$ \alpha(t)=\Ad_{h_t}\com\Ad_{g_{0}^{-1}}\com\phi_\ast(\gamma(0))
   =\Ad_{h_t g_{0}^{-1}}\com\phi_\ast(\rho) $$
and this implies the first part of the statement.
The second part of the statement follows along the same lines.

(ii) We know that $\K(\GG)$ is a locally
trivial bundle of Lie groups over $\GG_0$, so
$\omega:\hom(K,\K(\GG))\to\GG_0$ is also a locally trivial bundle
and therefore a Serre fibration. The same is true for the
restriction of $\omega$ to a $\GG$-invariant union of path-connected
components of $\hom(K,\K(\GG))$. 
On the other hand, the space
$\hom(K,\K(\GG))$ is a union of path-connected components of
$\hom(K,\GG)$ by (i). 
\end{proof}

As a consequence of Proposition \ref{propprg},
we can apply Theorem \ref{lexcon}, to obtain a long
exact sequence of homotopy groups for any
proper regular $K$-pointed Lie groupoid $(\GG,a,\rho)$
such that $\rho(K)\subset\K_a(\GG)$. However, if
$\rho$ is an isomorphism onto $\K_a(\GG)$, the long exact sequence
could be written in a simpler form:

\begin{theo}\label{lexprl}
Let $K$ be a compact Lie group,
let $\GG$ be a connected proper regular Lie groupoid
with basepoint $a$ and let
$\rho:K\to\K_a(\GG)$ be an isomorphism. We have a long exact sequence
\begin{equation*}
\begin{split}
\ldots\to\pi_n(\Aut(\K_a(\GG)))\to\pi_n^K(\GG)
\stackrel{\pi_{n}(\omega)}{\to}&\pi_n(\GG)
\stackrel{\partial}{\to}\pi_{n-1}(\Aut(\K_a(\GG)))\to\ldots \\
\ldots\to\pi_1(\Aut(\K_a(\GG)))\to\pi_1^K(\GG)
\stackrel{\pi_{1}(\omega)}{\to}&\pi_1(\GG)
\stackrel{\partial}{\to}\pi_{0}(\Aut(\K_a(\GG)))
\end{split}
\end{equation*}
which is natural with respect to Morita equivalences.
\end{theo}
\Rem
Here we used the isomorphism $\rho$ to identify
the space of isomorphisms from $K$ to $\K_a(\GG)$ with
$\Aut(\K_a(\GG))$.
Note that
$\pi_n(\Aut(\K_a(\GG)),\id)=\pi_n(\Aut(\K_a(\GG)))
=\pi_n(\Inn(\K_a(\GG)))=\pi_n(\K_a(\GG)/\Centralizer(\K_a(\GG)))$
for $n\geq 1$.
The sequence ends with the homomorphism
$$ \partial:\pi_1(\GG,a)\to\pi_{0}(\Aut(\K_a(\GG)))\;, $$
which we call the {\em monodromy map}. This homomorphism
does not depend on the choice of $\rho$, thus it is a
Morita invariant of $(\GG,a)$.
To understand the monodromy map geometrically,
one should represent an element of $\pi_1(\GG,a)$ by
a $\GG$-cocycle on $I$, \textit{i.e.}\ a series of paths in $\GG_0$
connected by arrows of $\GG$ (see \cite{MoerdijkMrcun2005}).
Then the monodromy is obtained by a combination of lifts
along the paths and adjoint actions of the connecting arrows.

\begin{proof}
This follows from Theorem \ref{lexcon}
and Proposition \ref{propprg}.
\end{proof}

\begin{ex}\rm
Let $\GG$ be a connected proper \'{e}tale Lie groupoid.
Any such groupoid is regular, has discrete orbits
and discrete isotropy groups. We have the associated
effect homomorphism $\GG\to\Eff(\GG)$ \cite{MoerdijkMrcun2003}
and $\K(\GG)$ is exactly its kernel.
By Theorem \ref{lexprl}
we have the long exact sequence
$$
\ldots\to\pi_n(\Aut(\K_a(\GG)))\to\pi_n^K(\GG)
\stackrel{\pi_{n}(\omega)}{\to}\pi_n(\GG)
\stackrel{\partial}{\to}\pi_{n-1}(\Aut(\K_a(\GG)))\to\ldots 
$$
for a choice of a point $a\in\GG_0$ and
an isomorphism of Lie groups $\rho:K\to\K_a(\GG)$.
Since the isotropy groups of $\GG$ are discrete,
the space $\Aut(\K_a(\GG))$ is discrete as well. The long exact sequence thus
implies that
$$ \pi_{n}(\omega):\pi_n^K(\GG)\to \pi_n(\GG) $$
is an isomorphism for any $n\geq 2$ and that we have an exact sequence
$$
1\to\pi_1^K(\GG)
\stackrel{\pi_{1}(\omega)}{\to}\pi_1(\GG)
\stackrel{\partial}{\to}\Aut(\K_a(\GG))\;. 
$$
\end{ex}

\begin{ex}\label{exLG}\rm
Let $G$ be a connected compact Lie group acting on a connected manifold
$M$, and assume that this action is regular, \textit{i.e.}\ that
the isotropy groups of this action have constant dimension.
Therefore, the associated translation groupoid
$G\ltimes M$ is regular and proper.
The isotropy group $G_x=(G\ltimes M)_x$
is a subgroup of $G$, for any $x\in M$.
Choose a basepoint $x\in M$, write $\K_x=\K_x(G\ltimes M)$
and let $\rho:K\to\K_x$ be
an isomorphism of Lie groups.
Note that $\K_x$ is a normal subgroup of $G_x$
and in particular a subgroup of $G$.
By Theorem \ref{lexprl}
we have the long exact sequence
$$
\ldots\to\pi_n(\Aut(\K_x))\to\pi_n^K(G\ltimes M)
\stackrel{\pi_{n}(\omega)}{\to}\pi_n(G\ltimes M)
\stackrel{\partial}{\to}\pi_{n-1}(\Aut(\K_x))\to\ldots 
$$
This sequence ends with the monodromy map
$$ \partial:\pi_1(G\ltimes M,x)\to\pi_{0}(\Aut(\K_x))\;.$$
Note that, by construction and by
\cite[Lemma 38.1]{ConnerFloyd},
we have $\partial(\pi_1(G\ltimes M,x))\subset\pi_0(\W_G(\K_x))$,
where $\W_G(\K_x)$  
is the subgroup of $\Aut(\K_x)$ consisting
of all automorphisms of $\K_x$ which are the restriction of
an inner automorphism of $G$. The group $\W_G(\K_x)$ 
is isomorphic to the quotient
$\Normalizer_G(\K_x)/\Centralizer_G(\K_x)$
of the normalizer of $\K_x$ in $G$
and the centralizer of $\K_x$ in $G$.

We may choose a faithful unitary representation of the
compact Lie group $G$.
This representation restricts to a faithful unitary
representation $r:\K_x\to U(m)$ of $\K_x$.
We write $\W_{r}(\K_x)$ for the subgroup of
$\Aut(\K_x)$ which is identified with the group
$\W_{\U(m)}(r(\K_x))$ by $r$.
Note that $\W_G(\K_x)\subset\W_{r}(\K_x)$, so
$\partial(\pi_1(G\ltimes M,x))\subset\pi_0(\W_G(\K_x))$
implies
$$ \partial(\pi_1(G\ltimes M,x))\subset\pi_0(\W_r(\K_x))\;.$$
\end{ex}

\subsection{Obstruction to presentability}
For a given proper Lie groupoid $\GG$ it is a natural question
to ask whether it is presentable as a global quotient,
or more precisely, whether it is
Morita equivalent to the translation groupoid
associated to an action of a compact connected Lie group.
It is well-known that any effective \'{e}tale proper Lie groupoid
is Morita equivalent to the translation groupoid of an action
of a unitary group on a manifold (such groupoids represent effective
orbifolds).
On the other hand,
there exist particular examples of connected proper regular
Lie groupoids which are not presentable in this way,
while it is not known yet
if any proper \'{e}tale Lie groupoid is presentable.

Let $\GG$ be a connected proper regular Lie groupoid
with basepoint $a\in \GG_0$. The associated
monodromy map
$$ \partial:\pi_1(\GG,a)\to\pi_0(\Aut(\K_a(\GG)))$$
is natural with respect to Morita equivalences.
In particular, if $\GG$ is Morita equivalent to
the translation groupoid $G\ltimes M$ of an action of
a connected compact Lie group $G$ on a connected manifold $M$,
then the groupoid $G\ltimes M$ is regular
and we have the commutative diagram
$$
\xymatrix{
\pi_1(\GG,a) \ar[d]_{\pi_1(P)} \ar[r]^-\partial & 
             \pi_0(\Aut(\K_a(\GG))) \ar[d]^-{\pi_0(\Ad_{P_p})}\\
\pi_1(G\ltimes M,x)   \ar[r]^-\partial  & \pi_0(\Aut(\K_x))
}
$$
where $P$ is a Morita equivalence from $\GG$ to $G\ltimes M$,
$p$ is a point in $P$ such that $a=\pi(p)$, $x=\epsilon(p)$
and $\K_x=\K_x(G\ltimes M)\subset G_x\subset G$.
By Example \ref{exLG} we know that
$\partial(\pi_1(G\ltimes M,x))\subset \pi_0(\W_G(\K_x))$, and
we obtain the following obstruction to presentability:

\begin{theo}\label{obstruction}
Let $\GG$ be a connected proper regular Lie groupoid
with basepoint $a\in \GG_0$. If $\GG$ is Morita equivalent
to the translation groupoid associated to an action of a connected
compact Lie group on a smooth manifold, then there exists
a faithful unitary representation $r:\K_a(\GG)\to \U(m)$
such that $\partial(\pi_1(\GG,a))\subset \pi_0(\W_{r}(\K_a(\GG)))$.
\end{theo}

\Rem
Observe that the condition
$\partial(\pi_1(\GG,a))\subset\pi_0(\W_{r}(\K_a(\GG)))$
is satisfied if and only if
the character of the representation $r$ is fixed under precomposition
with any automorphism representing a class in the image
of the monodromy map $\partial:\pi_1(\GG,a)\to\pi_0(\Aut(\K_a(\GG)))$.

\begin{ex}\rm
Let $G$ be a Lie group, let $M$ be a connected manifold with
basepoint $a$ and let $R:\pi_1(M,a)\to\Aut(G)$ be a homomorphism.
Choose a universal covering space $q:\tilde{M}\to M$ and
basepoint $\tilde{a}\in q^{-1}(a)$. We have a natural right action
of $\pi_1(M,a)$ on $\tilde{M}$ by deck transformations. Associated to $R$
we have a Lie groupoid $\BB$ over $M$ with space of arrows
$$ \BB_1=\tilde{M}\times_{\pi_1(M,a)}G $$
which is the space of orbits of the diagonal $\pi_1(M,a)$-action on
$\tilde{M}\times G$. With source and target map of $\BB$ both being equal
to the natural projection to $M$, the Lie groupoid $\BB$ is  in fact
a locally trivial bundle of Lie groups over $M$. The long exact sequence of
the Serre fibration $M\to\BB$ gives us the short exact sequence
$$ 1 \to
   \pi_n(M) \to
   \pi_n(\BB) \to
   \pi_{n-1}(G) \to
   1
$$
for any $n\geq 1$. In particular, the group $\pi_1(M,a)$ is
a subgroup of $\pi_1(\BB,a)$.

The Lie groupoid $\BB$ is
regular with $\I(\BB)=\K(\BB)=\BB$. Since
$\BB$ is a locally trivial bundle of Lie groups,
it follows that
$\omega:\hom(K,\BB)\to M$
is also a locally trivial bundle and hence a Serre fibration,
so $\omega:\BB^K\to\BB$ is a Serre fibration, for any Lie group $K$.
For any choice of a $K$-basegroup $\rho$ in $\BB$
we therefore have the long exact sequence
$$
\ldots\to\pi_n(\hom(K,\BB))\to\pi_n^K(\BB)
\stackrel{\pi_{n}(\omega)}{\to}\pi_n(\BB)
\stackrel{\partial}{\to}\pi_{n-1}(\hom(K,\BB))\to\ldots
$$
In special, if $G$ is compact, $K=G$ and
$\rho:G\to\BB_a$ is the natural isomorphism of Lie groups,
we get the long exact sequence 
$$
\ldots\to\pi_n(\Aut(\BB_a))\to\pi_n^G(\BB)
\stackrel{\pi_{n}(\omega)}{\to}\pi_n(\BB)
\stackrel{\partial}{\to}\pi_{n-1}(\Aut(\BB_a))\to\ldots
$$
At the end of this sequence we have the map
$\partial:\pi_1(\BB,a)\to\pi_0(\Aut(\BB_a))$. The restriction
of this homomorphism to $\pi_1(M,a)$ is exactly the composition
of $R$ with the natural quotient map
$\Aut(G)\to\pi_0(\Aut(G))\cong\pi_0(\Aut(\BB_a))$.

If the Lie groupoid $\BB$ is presentable as a global quotient,
the previous theorem implies that there exists a faithful
representation $r:G\to U(m)$ such that
$\partial(\pi_1(\BB,a))\subset \pi_0(\W_{r}(G))$.
In particular, this implies
$$ R(\pi_1(M,a))\subset \W_{r}(G)\;. $$
In other words, if we choose the homomorphism $R$ so that
this condition is not satisfied for any
faithful representation $r$, then $\BB$ is not presentable.
Finding out if such $R$ exists is not too difficult when we know the group $G$
well enough, \textit{e.g.}\ if we explicitly know all of its irreducible
characters and all characters of its faithful representations.
For example, it is easy to find such $R$ for the case
when $G$ is the $n$-torus, $n\geq 2$ (this example
of non-presentable proper Lie groupoid is known for $n=2$, 
see \textit{e.g.}\ \cite{LuckOliver,Trentinaglia}).
\end{ex}

\end{document}